\documentclass{amsart}
\usepackage{amssymb}

\usepackage{enumerate}
\usepackage[all,arc]{xy}
\usepackage{amssymb, amsfonts, lacromay}
\usepackage{mathrsfs,color,comment}
\usepackage{url}
\usepackage{tikz}
\usepackage{tikz-3dplot}
\usetikzlibrary{calc, intersections, decorations.pathmorphing, shapes,
  backgrounds, matrix, arrows}

\newcommand{\Set}{\mathbf{Set}}
\newcommand{\sSet}{\mathbf{sSet}}
\newcommand{\Cat}{\mathbf{Cat}}
\newcommand{\Pos}{\mathbf{Pos}}
\newcommand{\Top}{\mathbf{Top}}

\newcommand{\arDiv}[3][$*$]{ \path (A#2) -- (A#3) node[midway] (A#2-#3) {#1};%
                        \draw[mainTri] (A#2) -- (A#2-#3); %
                        \draw[mainTri] (A#3) -- (A#2-#3);}
\newcommand{\baryc}[4][$*$]{ \node (A#2-#3-#4) at (barycentric cs:A#2=1,A#3=1 ,A#4=1) {#1};%
                             \draw[mainTri] (A#2-#3) -- (A#2-#3-#4);
                             \draw[mainTri] (A#2-#4) -- (A#2-#3-#4);
                             \draw[mainTri] (A#3-#4) -- (A#2-#3-#4);}

\newcommand{\C}{\mathcal{C}}
\newcommand{\W}{\mathcal{W}}

\tikzset{mainTri/.style={densely dotted, ->, >=to}}

\DeclareMathOperator{\Sd}{Sd}
\DeclareMathOperator{\Ex}{Ex}
\renewcommand{\tilde}{\widetilde}
\renewcommand{\sU}{\mathcal{U}}
\renewcommand{\sV}{\mathcal{V}}

\newtheorem{thm}{Theorem}[section]
\newtheorem{cor}[thm]{Corollary}
\newtheorem{prop}[thm]{Proposition}
\newtheorem{lem}[thm]{Lemma}

\theoremstyle{definition}
\newtheorem{defn}[thm]{Definition}

\newtheorem{exmp}[thm]{Example}

\newtheorem{notn}[thm]{Notation}

\newtheorem{rem}[thm]{Remark}

\makeatletter
\let\c@equation\c@thm
\makeatother
\numberwithin{equation}{section}

\makeatletter
\ifx\SK@label\undefined\let\SK@label\label\fi
 \let\your@thm\@thm
 \def\@thm#1#2#3{\gdef\currthmtype{#3}\your@thm{#1}{#2}{#3}}
 \def\mylabel#1{{\let\your@currentlabel\@currentlabel\def\@currentlabel
  {\currthmtype~\your@currentlabel}
 \SK@label{#1@}}\label{#1}}
 \def\myref#1{\ref{#1@}}
\makeatother

\newcommand{\n}{\mathbf{n}}

\begin{document}

\begin{abstract} Let $G$ be a discrete group. We prove that the category of
  $G$-posets admits a model structure that is Quillen equivalent to the standard
  model structure on $G$-spaces.  As is already true nonequivariantly, the three
  classes of maps defining the model structure are not well understood
  calculationally.  To illustrate, we exhibit some examples of cofibrant and
  fibrant posets and an example of a non-cofibrant finite poset.
\end{abstract}

\title{The homotopy theory of equivariant posets}
\author{Peter May}
\address{Department of Mathematics, University of Chicago, Chicago, IL 60637, USA}
\email{may@math.uchicago.edu}
\author{Marc Stephan}
\thanks{The second author was supported by SNSF grant 158932.}
\address{Department of Mathematics, University of British Columbia, Vancouver, BC V6T 1Z2, Canada}
\email{mstephan@math.ubc.ca}
\author{Inna Zakharevich}
\thanks{The third author was supported by an NSF MSRFP grant.}
\address{Department of Mathematics, Cornell University, Ithaca, NY  14853-4201, USA}
\email{zakh@math.cornell.edu}

\maketitle

\tableofcontents

\section{Introduction} 
In \cite{thomason80}, Thomason proved that categories model the homotopy theory
of topological spaces by proving that the category $\Cat$ of (small) categories
has a model structure that is Quillen equivalent to the standard model structure
on the category $\Top$ of topological spaces.  In \cite{raptis}, Raptis proved
that the category of posets also models the homotopy theory of topological
spaces by showing that the category $\Pos$ of posets has a model structure that
is Quillen equivalent to the Thomason model structure on $\Cat$.  It is natural
to expect this to hold since Thomason proved in
\cite[Proposition~5.7]{thomason80} that cofibrant categories in his model
structure are posets.  The first and third authors rediscovered this, observing
that a geodesic proof, if not the statement, of that result is already contained
in Thomason's paper. This implies that all of the algebraic topology of spaces
can in principle be worked out in the category of posets.  It can also be viewed
as a bridge between the combinatorics of partial orders and algebraic topology.

In this paper we prove an analogous result for the category of $G$-spaces for a
discrete group $G$.  For a category $\aC$, let $G\aC$ denote the category of
objects with a (left) action of $G$ and maps that preserve the action.  In
\cite{manyauthors}, Bohmann, Mazur, Osorno, Ozornova, Ponto, and Yarnall proved
in precise analogy to Thomason's result that $G\Cat$ models the homotopy theory
of $G$-spaces.  Here we prove the pushout of the results of Raptis and Bohmann, et al:
the category $G\Pos$ of $G$-posets admits a model structure that is Quillen equivalent 
to the model structure on the category $G\Cat$ of $G$-categories and therefore also 
Quillen equivalent to the model structures on $G\sSet$ and $G\Top$.  Just as the model 
structure on $\Pos$ is implicit in Thomason's paper \cite{thomason80}, we shall see that 
the model structure on $G\Pos$ is implicit in the six author paper \cite{manyauthors}.

While the background makes this an expected result, it is perhaps surprising, at least psychologically.
There is relatively little general study of equivariant posets in either the combinatorial or topological
literature, especially not from a homotopy theoretic perspective.  One thinks of group actions as permutations, as exemplified by the symmetric groups, and it does not come naturally to think of a 
general theory of groups acting by order-preserving maps of posets.  
However, our theorem says that group actions on posets abound: every $G$-space is weakly equivalent to the classifying $G$-space of a $G$-poset, where a map $f$ of $G$-spaces is a weak equivalence if its fixed point maps $f^H$ are weak equivalences for all subgroups $H$ of $G$. The result can be viewed as a formal bridge between equivariant combinatorics and equivariant algebraic topology.

The combinatorial literature seems to start with Stanley's paper \cite{stanley},
which restricts to finite posets and focuses on the connection with
representation theory.  A paper of Babson and Kozlov \cite{babsonkozlov} about
$G$-posets $X$ focuses on problems arising from the fact that the orbit category
$X/G$ is generally not a poset.  There is considerable group theory literature
about posets of subgroups of $G$ with $G$ acting by conjugation, starting from
Quillen's paper \cite{quillen}.  That led Th\'evenaz and Webb to an equivariant
generalization of Quillen's Theorem A applicable to $G$-posets
\cite{thevenazwebb}. In turn, that led to Welker's paper \cite{welker}, which
considers the order $G$-complex associated to a $G$-poset, again with group
theoretic applications in mind.

Let $\aO_G$ denote the orbit category of $G$. Its objects are the $G$-sets $G/H$ and its morphisms are the 
$G$-maps.  Just as for $G$-spaces, $G$-simplicial sets (that is, simplicial $G$-sets), and $G$-categories, it is natural to start with the levelwise (or projective) model structure on the category $\aO_G$-$\Pos$ of contravariant functors $\aO_G\rtarr \Pos$.  As a functor category, $\aO_G$-$\Pos$ inherits a model structure from $\Pos$. Its fibrations and weak equivalences are defined levelwise.  It is standard that this gives a compactly generated model 
structure (e.g. \cite[11.6.1]{hirschhorn}).\footnote{Compactly generated is a variant of cofibrantly generated that applies when only countable 
colimits are needed in the small object argument, that is, when transfinite colimits are 
unnecessary and irrelevant, as they are in all of the model structures we shall consider.
This variant is discussed in detail in \cite[\S15.2]{mayponto}.  It seems reasonable to eliminate
transfinite verbiage whenever possible, and that would shorten and simplify some of the work in the 
sources we shall cite.}

Define the fixed point diagram functors 
\[\PH\colon G\Pos\rtarr \aO_G \text{-}  \Pos\ \ \ \tand \ \ \ \PH\colon G\Cat\rtarr \aO_G \text{-} \Cat \] 
by
\[ \PH(X)(G/H) = X^H.\]  These functors $\PH$ have left adjoints, denoted $\LA$; in both cases, $\LA$ sends 
a contravariant functor $Y$ defined on $\aO_G$ to $Y(G/e)$. 

We prove that $G\Pos$ inherits a model 
structure from $\aO_G$-$\Pos$. The analogue for $G\Cat$ is \cite[Theorem A]{manyauthors}.  
After recalling details of the model structures already cited, we shall prove the 
following theorem.

\begin{thm}\mylabel{Main} The functor $\PH$ creates a compactly generated proper model 
structure on $G\Pos$, so that a map of $G$-posets is a weak equivalence or fibration if 
it is so after applying $\PH$. The adjunction $(\LA,\PH)$ is a Quillen equivalence between 
$G\Pos$ and $\aO_G$-$\Pos$.
\end{thm}

Replacing $\Pos$ with $\Cat$ in \myref{Main} gives the statement of \cite[Theorem A]{manyauthors}. The 
strategy of proof in \cite{manyauthors} is to verify general conditions on a model
category $\aC$ that ensure that $G\aC$ inherits a model structure from $\aO_G$-$\aC$.\footnote{There
are two slightly different ways to equip $G\aC$ with a model structure, either transferring the model 
structure from $\aO_G$-$\aC$, as we shall do, or from copies of $\aC$ via all of the fixed point functors, as
in \cite{manyauthors, stephan}.}
The cited general conditions are taken from a paper of the second author \cite{stephan}.
Our proof of \myref{Main} will proceed in the same way. The following result is a formal
consequence of \myref{Main} and its analogue for $\Cat$. 

\begin{thm}\mylabel{MainToo}  The adjunction $(P,U)$ between $G\Cat$ and $G\Pos$ is a Quillen equivalence.  
Therefore, $G\Pos$ is Quillen equivalent to $G\sSet$ and $G\Top$. 
\end{thm}

The following diagram displays the relevant equivariant Quillen equivalences.

\[ \xymatrix{
G\Top \ar@<-.5ex>[rr]_-{S_*} \ar@<.5ex>[dd]^{\PH}  & & 
G\sSet \ar@<-.5ex>[ll]_-{|-|} \ar@<.5ex>[rr]^-{\PI \Sd^2}  \ar@<.5ex>[dd]^{\PH} & & 
G\Cat \ar@<.5ex>[ll]^-{\Ex^2 N}  \ar@<.5ex>[rr]^-{P} \ar@<.5ex>[dd]^{\PH} & & 
G\Pos \ar@<.5ex>[ll]^-{U} \ar@<.5ex>[dd]^{\PH} \\ 
& & & \\
\aO_G\text{-}\Top \ar@<-.5ex>[rr]_-{S_*} \ar@<.5ex>[uu]^{\LA}  & & 
\aO_G\text{-}\sSet \ar@<-.5ex>[ll]_-{|-|} \ar@<.5ex>[rr]^-{\PI \Sd^2} \ar@<.5ex>[uu]^{\LA} & & 
\aO_G\text{-}\Cat \ar@<.5ex>[ll]^-{\Ex^2 N}  \ar@<.5ex>[rr]^-{P} \ar@<.5ex>[uu]^{\LA}  & & 
\aO_G\text{-}\Pos \ar@<.5ex>[ll]^-{U} \ar@<.5ex>[uu]^{\LA} \\}  \]
The definitions of $\PI$, $\Sd$, $\Ex$, and $N$ are recalled in the next section.

All of the vertical adjunctions and the adjunctions on the bottom row are Quillen
equivalences, hence so are all of the adjunctions on the top row. Applied to the
righthand square, this gives the proof of \myref{MainToo}.  Applied to the middle 
square, this gives \cite[Theorem B]{manyauthors}, which is the equivariant version of
Thomason's comparison between $\sSet$ and $\Cat$. 

\begin{rem}  Both equivariantly and nonequivariantly, replacing $\Cat$ by $\Pos$
ties in the Thomason model structure to more classical algebraic topology.  The 
composite $N\com U\colon \Pos\rtarr \sSet$ coincides with the composite of the
functor that sends a poset to its order complex and the canonical functor from ordered 
simplicial complexes to simplicial sets, and the same is true equivariantly.  It 
also ties in the Thomason model structure to finite $T_0$-spaces and, more generally
$T_0$-Alexandroff spaces, or $A$-spaces, since the categories of posets and
$A$-spaces are isomorphic.
\end{rem}

An interesting and unfortunate feature of all of the model structures discussed
in this paper is that the classes of weak equivalences, cofibrations, and
fibrations are defined formally, using non-constructive arguments.  In no case
do we have a combinatorially accessible description of any of these classes of
maps. Even in the case when $G$ is trivial very little is known about the
structure.  In \cite[Theorem 2.2.11]{cisinski04}, Cisinski gives a
characterization of the subcategory of weak equivalences in $\Cat$ through a
global characterization,  but that does not allow us to determine
whether or not a particular morphism is a weak equivalence.

The state of the art for fibrant and cofibrant objects is similarly sparse.   
The problem of determining the cofibrant posets has recently been studied by
Bruckner and Pegel \cite{BP}, who show in particular that every poset with at
most five elements is cofibrant.  In \S\ref{Cof},  we prove that all finite posets of dimension 
one are cofibrant and give an example of a
six element poset that is {\em not} cofibrant.\footnote{Amusingly, when we found
  this example we did not know that it is the smallest possible one.}
  
The problem of determining the fibrant categories has recently been studied by Meier and Ozornova
\cite{MO}.  In \S\ref{Fib}, we use work of Droz and the third author \cite{DZ} to obtain a more concrete
understanding of the posets that the main theorem of \cite{MO} shows to be fibrant.

Before turning to the equivariant generalizations, we review and reprove the
nonequivariant theorems, giving some new details that streamline and clarify the
key arguments.

\section{Background} We recall as much as we need about the definitions of the nonequivariant versions 
of the functors in the diagram above and describe the relevant nonequivariant model structures. 
 Of course, the nerve $N\aC$ of a category $\aC$ is the 
simplicial set with 
\[(N\aC)_n = \big\{x_0 \rtarr \cdots \rtarr x_n \in \aC\big\}.\] Define $(\Sd \DE)(n)$
to be the nerve of the poset of nonempty subsets of $\{0,1,\cdots,n\}$. Then
$\Sd\DE$ is a covariant functor $\DE \rtarr \sSet$.  Let $K\colon \DE^{op}\rtarr
\Set$ be a simplicial set.  The subdivision $\Sd K$ is the simplicial set
defined conceptually as the tensor product of functors (given by the evident
left Kan extension)
\[\Sd K = K\otimes_{\DE} \Sd\DE.\]
The functor $\Ex$ is the right adjoint of $\Sd$; we will not need a description of it.

The fundamental category\footnote{Following \cite{thomason80}, the functor $\PI$ is generally denoted $c$, 
or sometimes $cat$, in the literature.} $\PI K$ has object set $K_0$ and morphism set freely generated by $K_1$,
where $x\in K_1$ is viewed as a morphism $d_1x\rtarr d_0x$, subject to the relations
\[ d_1y = (d_0y)\com (d_2y)   \  \ \hbox{for each $y\in K_2$}\ \ \ \text{and}\ \  \ 
s_0x=\id_x \  \ \hbox{for each $x\in K_0$.}\]

The functor $U$ is the full and faithful functor that sends a poset $X$ to $X$ 
regarded as the category with objects the elements of $X$ and a morphism $x\rtarr y$ 
whenever $x\leq y$. The image of $U$ consists of skeletal categories with at most one
morphism $x\rtarr y$ for each pair of objects $(x,y)$. The functor $P$ sends a category $\aC$ 
to the poset $P\aC$ with points
the equivalence classes $[c]$ of objects of $\aC$, where $c\sim d$ if there are morphisms 
$c \rtarr d$ and $d \rtarr c$ in $\aC$. The partial order $\leq$ is defined by $[c]\leq [d]$ if there is a 
morphism $c\rtarr d$ in $\aC$, a condition independent of the choice of representatives in
the equivalence classes.  Note crucially that $P\com U$ is the identity functor.  We often drop the 
notation $U$, regarding posets as categories.

We recall the specification of the model structures that we are starting from.

\begin{defn}  A functor $F\colon \aC \rtarr \aD$ between (small) categories is a fibration or weak equivalence if $\Ex^2 NF$ is a fibration or weak equivalence.  An order preserving function $f\colon X \rtarr Y$ between posets is a fibration or weak equivalence if $Uf$ is a fibration or weak equivalence; that is, $f$ is a fibration or weak equivalence if it is so when considered as a functor.
\end{defn} 

As noted by Thomason \cite[Proposition 2.4]{thomason80}, $F$ is a weak equivalence if and only if $NF$ is a weak equivalence.

\begin{notn} Let $\aI$ denote the set of generating cofibrations $\pa \DE[n] \rtarr \DE[n]$ and let $\aJ$ denote the set of generating acyclic cofibrations $\LA^k[n]\rtarr \DE[n]$ for the standard model structure on $\sSet$.  
\end{notn}

\begin{thm}[Thomason]\mylabel{Thom}  With these fibrations and weak equivalences, $\Cat$ is a 
compactly generated proper model category whose sets of generating cofibrations and generating acyclic cofibrations are 
$\PI\Sd^2\aI $ and $\PI\Sd^2\aJ$. Via the adjunction $(\PI\Sd^2,\Ex^2 N)$, this model structure is Quillen equivalent to the standard model structure on $\sSet$.
\end{thm}

\begin{rem} In contrast to more recent papers, which use but do not always need transfinite colimits,
Thomason's paper preceded the formal introduction of cofibrantly generated model categories, and he 
neither used nor needed such colimits; our statement is a reformulation of what he actually proved.
\end{rem}

\begin{thm}[Raptis]\mylabel{OldMain} With these fibrations and weak equivalences, $\Pos$ is a compactly generated
proper model category whose sets of generating cofibrations and generating acyclic cofibrations are 
$P\PI\Sd^2\aI $ and $P\PI\Sd^2\aJ$.  Via the adjunction $(P,U)$, this model structure is Quillen equivalent to the Thomason model structure 
on $\Cat$.
\end{thm} 

\section{The proofs of Theorems \ref{Thom} and \ref{OldMain}}

The proofs of the model axioms in \cite{thomason80, raptis} can be streamlined by use of a slight variant of Kan's
transport theorem \cite[Theorem 11.3.2]{hirschhorn}.  It is proven in \cite[16.2.5]{mayponto}.

\begin{thm}[Kan]\mylabel{Kan}
  Let $\aC$ be a compactly generated model category with generating
  cofibrations $\aI$ and generating acyclic cofibrations $\aJ$.  Let $\aD$ be a
  bicomplete category, and let $F\colon \aC \rightleftarrows \aD :\!U$ be a pair
  of adjoint functors.  Assume that
  \begin{itemize}
  \item[(i)] all objects in the sets $F\aI $ and $F\aJ  $ are compact and
  \item[(ii)] the functor $U$ takes relative $F\aJ  $-cell complexes to weak equivalences.
  \end{itemize}
  Then there is a compactly generated model structure on $\aD$ such that $F\aI $ is
  the set of generating cofibrations, $F\aJ  $ is the set of generating acyclic
  cofibrations, and the weak equivalences and fibrations are the morphisms
  $f$ such that $Uf$ is a weak equivalence or fibration.  Moreover, $(F\dashv U)$ is a Quillen pair. 
\end{thm}

\begin{rem} It is clear that if $\aC$ is right proper then so is $\aD$.  
Since the standard model structure on $\sSet$ is right proper, so are the 
model structures on $\Cat$ and $\Pos$ described below.  It is less clear
that they are left proper, as we shall discuss.
\end{rem} 

Compactly generated makes sense when the generating sets are compact in the sense of 
\cite[15.1.6]{mayponto}, as we require in condition (i).  In \myref{OldMain}, the domain 
posets of all maps in $P\PI \Sd^2\aI  $ and $P\PI \Sd^2\aJ$ are finite since they are obtained 
from simplicial sets with only finitely many $0$-simplices. Therefore they are compact relative 
to all of $\Cat$ and in particular are compact relative 
to $P\PI\Sd^2\aI $ and $P\PI\Sd^2\aJ$. This shows that (i) holds, and we need only 
prove (ii) to complete the proof of the model axioms in \myref{OldMain}. 

Since we are working with compact generation, a relative $P\PI\Sd^2\aJ$-complex 
$i\colon A\rtarr X = \colim X_n$ is the colimit of a sequence 
of maps of posets $X_n\rtarr X_{n+1}$, where $X_0 = A$ and $X_{n+1}$ is a pushout 
\begin{equation}\label{push} 
\xymatrix{  K_n\ar[d]_{j} \ar[r]^-{f} & X_n \ar[d] \\
    L_n\ar[r]  \ar[r] & X_{n+1}\\}
\end{equation}
in $\Pos$ in which $j$ is a coproduct of maps in $P\PI\Sd^2\aJ$.
We must prove that such a map $i$, or rather $Ui$, is a weak equivalence in $\Cat$.
The only subtlety in the proof of \myref{OldMain} is that pushouts in $\Cat$ between maps in 
$\Pos$ are generally not posets. Rather, pushouts in $\Pos$ are constructed by taking pushouts 
in $\Cat$ and then applying the left adjoint $P$.  However, results already in \cite{thomason80} show
that we do not encounter that problem when constructing relative $P\PI\Sd^2\aJ$-complexes, as 
we now explain.   

To deal with pushouts when proving \myref{Thom}, Thomason introduced the notion of a Dwyer map.

\begin{defn} Let $\aS$ be a subcategory of a category $\aC$. Then $\aS$ is
  called a {\em sieve in $\aC$} if for every morphism 
$f\colon c\rtarr s$ in $\aC$ with $s\in\aS$,  
$c$ and $f$ are in $\aS$.  Dually, $\aS$ is a {\em cosieve} if for every morphism $f\colon s\rtarr c$ in $\aC$ with
$s\in\aS$,  $c$ and $f$ are in $\aS$.  In either case, $\aS$ must be a full subcategory of $\aC$. Observe that
if a sieve factors as a composite of inclusions $\aS\rtarr \aT\rtarr \aC$, then $\aS\rtarr \aT$ is again a sieve.
\end{defn}

\begin{defn} A functor $k\colon \aS\rtarr \aC$ in $\Cat$ or in $\Pos$ is a Dwyer map if $k$ is the 
inclusion of a sieve and $k$ factors as a composite 
\[  \xymatrix@1{\aS\ar[r]^-{i} & \aT \ar[r]^-{j} & \aC, \\} \] 
where $j$ is the inclusion of a cosieve and $i$ is an inclusion with a right adjoint 
$r\colon \aT\rtarr \aS$ such that the unit $\id\rtarr r\com i$ 
of the adjunction is the identity.
\end{defn}

The following sequence of results shows that \myref{OldMain} is directly implied by details in
Thomason's paper \cite{thomason80} that he used to prove \myref{Thom}.  Except that we add in the trivial statement about 
coproducts, the first is \cite[Lemma 5.6]{thomason80}. 

\begin{lem}\mylabel{cruxtoo} The following statements about posets hold.
\begin{enumerate}[(i)]
\item For any simplicial set $K$, $\PI \Sd^2 K$ is a poset. 
\item Any subcategory of a poset is a poset.
\item Any coproduct of posets in $\Cat$ is a poset.
\item If $j\colon K \rtarr L$ is a Dwyer map between posets and
$f\colon K \rtarr X$ is a map of posets, then the
pushout $Y$ in $\Cat$ of $j$ and $f$ is a poset.
\item The (directed) colimit in $\Cat$ of any sequence of maps of posets is a poset.
\end{enumerate}
\end{lem} 

The second is \cite[Proposition 4.2]{thomason80}.

\begin{lem}\mylabel{thomcrux} Let $K\subset L$ be an inclusion
of simplicial sets that arises from an inclusion of ordered 
simplicial complexes. Then the induced map $\PI \Sd^2K \rtarr \PI \Sd^2 L$
is a Dwyer map in $\Cat$ and therefore, by \myref{cruxtoo}(i), in $\Pos$. 
\end{lem} 

For completeness, we state an analogue to  \myref{cruxtoo} about Dwyer maps  in $\Cat$. 
It combines part of 
\cite[Proposition 4.3]{thomason80} with the correct parts of \cite[Lemma 5.3]{thomason80}.
We again add in a trivial statement about coproducts.

\begin{lem}\mylabel{thommore} The following statements about Dwyer maps in $\Cat$ hold.
\begin{enumerate}[(i)]
\item Any composite of Dwyer maps is a Dwyer map.
\item Any coproduct of Dwyer maps is a Dwyer map.
\item If $j\colon \aK \rtarr \aL$ is a Dwyer map and $f\colon \aK \rtarr \aC$ is a functor, then the
pushout $k\colon \aC\rtarr \aD$ of $j$ along $f$ is a Dwyer map.
\item For Dwyer maps $\aC_n\rtarr \aC_{n+1}$, $n\geq 0$, the induced map $\aC_0\rtarr \colim \aC_n$ 
is a Dwyer map.
\end{enumerate}
Therefore the same statements hold for Dwyer maps in $\Pos$. 
\end{lem}

\begin{cor}\mylabel{relcomp}  If $A$ is a poset and $i\colon A \rtarr X$ is a relative
$\PI \Sd^2\aJ$-complex in $\Cat$, then $X$ is a poset and $i$ is both a Dwyer 
map and a relative $\PI \Sd^2\aJ$-complex in $\Pos$. 
The same statement holds for relative $\PI \Sd^2\aI $-complexes. 
\end{cor}

\begin{rem}\mylabel{cofibrations}
Once the model structures on $\Pos$ and $\Cat$ are in place, the results above imply that a map $f$ between posets is a cofibration in $\Pos$ if and only if $f$ is a cofibration in $\Cat$.
\end{rem}

The real force of the introduction of Dwyer maps comes from the following result.
It combines Thomason's \cite[Proposition 4.3 and Corollary 4.4]{thomason80}.  

\begin{prop}\mylabel{key} If $j\colon \aK \rtarr \aL$ is a Dwyer map in $\Cat$, $f\colon \aK \rtarr \aC$ is a 
functor, and $\aD$ is their pushout, then the canonical map
\[  N\aL\cup_{N\aK}N\aC \rtarr N(\aL\cup_{\aK}\aC) = N\aD \]
is a weak equivalence. The same statement holds in $\Pos$.  Therefore, if $f$ is a 
weak equivalence, then so is the pushout $g\colon \aL\rtarr \aD$ of $f$ along $j$. 
\end{prop}

The last statement is inherited from the corresponding statement in $\sSet$. 

\begin{rem}\mylabel{cis} The incorrect part of \cite[Lemma 5.3]{thomason80} states that a retract 
of a Dwyer map is a Dwyer map.  As noticed by Cisinski \cite{cisinski99}, that is not true. He gave an example 
to show that a retract of a cofibration in $\Cat$ need not be a Dwyer map, which invalidates the proof 
that $\Cat$ is left proper given in \cite[Corollary 5.5]{thomason80}.  He introduced the slightly 
more general notion of a pseudo Dwyer map to get around this.  He proved that a retract of a pseudo Dwyer
map is a pseudo Dwyer map, so that any cofibration in $\Cat$ is a pseudo Dwyer map.   He then used that to
give a correct proof that $\Cat$ is left proper, and he observed that our Lemmas \ref{thomcrux} and \ref{thommore}
remain true with Dwyer maps replaced by pseudo Dwyer maps.
\end{rem}

The problem discussed in the remark does not arise when dealing with $\Pos$, where Dwyer maps and
pseudo Dwyer maps coincide, as follows directly from the definition of the latter.  Since we are omitting 
that definition, we give a simple direct proof of the following result.
Once the model structure is in place, it gives that cofibrations in $\Pos$ are Dwyer maps.
This highlights the technical convenience of posets, as compared with general categories.

\begin{lem}\mylabel{retract} A retract of a Dwyer map in $\Pos$ is a Dwyer map.  Therefore retracts in
$\Pos$ of relative $\PI\Sd^2\aI$-complexes are Dwyer maps.
\end{lem}
\begin{proof}
Consider the following diagram of posets, which commutes with $\si$ and $\ta$ omitted. 
All unlabeled arrows are inclusions. 
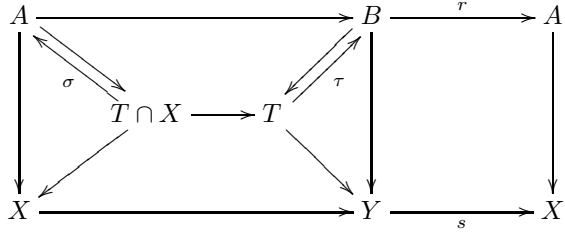
\[\xymatrix{
A \ar[rrr] \ar[dd] \ar@<.5ex>[dr] & & & B \ar[dd] \ar@<-.5ex>[dl] \ar[rr]^-{r} & &A \ar[dd]\\
& T\cap X \ar[dl]  \ar[r] \ar@<.5ex>[ul]^{\si} & T \ar[dr] \ar@<-.5ex>[ur]_{\ta}   &  & \\
X \ar[rrr] & & & Y \ar[rr]_-{s} & & X
\\} \]
We assume that $r$ restricts to the identity on $A$ and $s$ restricts to the identity on $X$.
We also assume that $B\rtarr Y$ is a sieve, $T\rtarr Y$ is a cosieve, and 
$\ta$ is right adjoint to the inclusion $B\rtarr T$ with unit the identity, so that $\ta$ restricts
to the identity on $B$.  We define $\si$ to be the restriction of $r\com \ta$ to $T\cap X$.  The 
following observations prove that $A\rtarr X$ is a Dwyer map.\\
(i) The restriction $T\cap X \rtarr X$ of the cosieve $T\rtarr Y$ is again a cosieve.\\
{\em Proof.} If $w\in T\cap X$ and $w\leq x$ in $X$, then $x\in T$, hence $x\in T\cap X$.\\ 
(ii) The restriction $A\rtarr X$ of the sieve $B\rtarr Y$ is again a sieve.\\
{\em Proof.}  If $a\in A$, $x\in X$, and $x\leq a$,
then $x\in B$ since $B\rtarr Y$ is a sieve, and then $x = s(x) = r(x) \leq r(a) = a$ in $A$.\\ 
(iii) $\si$ is right adjoint to the inclusion $A\rtarr T\cap X$, with unit the identity map.\\
{\em Proof.}  $\si$ restricts to the identity on $A$ since if $a\in A$, then 
$$\si(a) = (r\com \ta)(a) = r(a) = a.$$  
For the adjunction, we must
show that if $a\in A$ and $x\in T\cap X$, then $a\leq x$  if and only if $a \leq \si(x)$.  If $a\leq x$, then 
$a = \si(a) \leq \si(x)$.  Suppose $a \leq \si(x)$ and note that $\si(x) =(r\com \ta)(x) = (s\com \ta)(x)$. 
Since $\ta$ is right adjoint to $B\rtarr T$, the counit of the adjunction gives that $\ta(y) \leq y$ for any $y\in T$. 
Thus $(s\com \ta)(x) \leq s(x) = x$. 
\end{proof}

\begin{proof}[Proof of Theorems \ref{Thom} and \ref{OldMain}]   The heart of Thomason's proof of 
\myref{Thom} is the verification of condition (ii) of \myref{Kan}.  Since coproducts and colimits of weak
equivalences are weak equivalences, this reduces to showing that the pushouts in the construction
of relative $\aJ$-complexes are weak equivalences.  But that is immediate from \myref{key}.  Since 
a relative $P\PI\Sd^2\aJ$-complex in $\Pos$ is a special case of a relative $\PI\Sd^2\aJ$-complex in 
$\Cat$, condition (ii) of \myref{Kan} holds in $\Pos$ since it is a special case of the
condition in $\Cat$. This proves that $\Cat$ and $\Pos$ are compactly generated model categories. 
In view of \myref{retract}, 
\myref{key} also implies that $\Pos$ is left proper and therefore proper.  As pointed out in  \myref{cis},
Cisinski \cite{cisinski99} proves that $\Cat$ is left proper and therefore proper.

It remains to show that the adjunctions $(\PI \Sd^2,\Ex^2N)$ and $(P,U)$ are Quillen equivalences.  
To show that $(\PI \Sd^2,\Ex^2N)$ is a Quillen equivalence, it suffices to show that the composite 
$\Ex^2N$ induces an equivalence between the homotopy categories of $\Cat$ and $\sSet$. Quillen 
\cite[Ch.~VI, Corollaire 3.3.1]{illusie} proved that the nerve $N$ induces an equivalence.
Kan \cite[Ch.~III, Theorem 4.6]{goerssjardine} proved that $\Ex$ and therefore $\Ex^2$ induces an 
equivalence by showing that there is a natural weak equivalence $K\rtarr \Ex K$ for simplicial sets $K$.

To show that $(P,U)$ is a Quillen equivalence, it suffices to show  
that for all cofibrant categories $\aC\in \Cat$ and all fibrant posets $X\in \Pos$, a 
functor $f\colon \aC \rtarr U X$ is a weak equivalence if and only if its adjunct
$\tilde f\colon P\aC \rtarr X$ is a weak equivalence.  Since $\aC$ is cofibrant, 
it is a poset, hence $\aC = UY$ for a poset $Y$.  But then 
$U\tilde f = f$ and the conclusion holds by the definition of weak equivalences in $\Pos$. 
\end{proof}

\begin{rem} The fact that $\PI\Sd^2K$ is a poset for any simplicial set $K$ is closely related to the 
less well-known fact that $\Sd^2 \aC$ is a poset for any category $\aC$. However, the subdivision functor 
on $\Cat$ plays no role in Thomason's work or ours.  The relation between these subdivision functors is 
studied in \cite{delhoyo} and \cite{finite}. 
\end{rem}

\section{Equivariant Dwyer maps and cofibrations}

To mimic the arguments just given equivariantly, we introduce equivariant Dwyer maps and relate them to cofibrations in $\Pos$.

\begin{defn} A functor $k\colon \aS\rtarr \aC$ in $G\Cat$ or in $G\Pos$ is a Dwyer $G$-map 
if $k$ is the inclusion of a sieve and $k$ factors in $G\Cat$ as a composite 
\[  \xymatrix@1{\aS\ar[r]^-{i} & \aT \ar[r]^-{j} & \aC, \\} \] 
where $j$ is the inclusion of a cosieve and $i$ is an inclusion with a right adjoint 
$r\colon \aT\rtarr \aS$ in $G\Cat$ such that the unit $\id\rtarr r\com i$ 
of the adjunction is the identity.\footnote{Since the unit is the identity, the pair $(i,r)$ is automatically an adjunction in the $2$-category 
of $G$-objects in $\Cat$, equivariant functors, and equivariant natural transformations.}
\end{defn}

The following two lemmas are immediate from the definition.  

\begin{lem}\mylabel{fixDwyer} If $k$ is a Dwyer $G$-map, then $k^H$ is a Dwyer map for any subgroup $H$ of $G$.
\end{lem}

Regard the $G$-set $G/H$ as a discrete $G$-category (identity morphisms only).

\begin{lem} If $j\colon K\subset L$ is a Dwyer map and $H$ is a subgroup of $G$, then
$\id\times j\colon G/H\times K \rtarr G/H \times L$ is a Dwyer $G$-map. \end{lem}

We have the equivariant analogues of \myref{thommore} and \myref{relcomp}, with the same proofs.

\begin{lem}\mylabel{thommore2} The following statements about Dwyer $G$-maps in $G\Cat$ hold.
\begin{enumerate}[(i)]
\item Any composite of Dwyer $G$-maps is a Dwyer $G$-map.
\item Any coproduct of Dwyer $G$-maps is a Dwyer $G$-map.
\item If $j\colon \aK \rtarr \aL$ is a Dwyer $G$-map and $f\colon \aK \rtarr \aC$ is a $G$-map, then the
pushout $k\colon \aC\rtarr \aD$ of $j$ along $f$ is a Dwyer $G$-map.
\item For Dwyer $G$-maps $\aC_n\rtarr \aC_{n+1}$, $n\geq 0$, the induced map $\aC_0\rtarr \colim \aC_n$ 
is a Dwyer $G$-map.
\end{enumerate}
Therefore the same statements hold for Dwyer $G$-maps in $G\Pos$. 
\end{lem}

Let $G\PI \Sd^2\aI$ and $G\PI \Sd^2\aJ$ denote the sets of all $G$-maps that are of the form $\id\times j\colon G/H\times K \rtarr G/H \times L$,
where $j$ is in $\PI \Sd^2\aI$ or $\PI \Sd^2\aJ$.  These are the generating cofibrations and generating acyclic 
cofibrations in $G\Cat$.  

\begin{cor}\mylabel{relcomp2}  If $A$ is a $G$-poset and $i\colon A \rtarr X$ is a relative
$G\PI \Sd^2\aJ$-complex in $G\Cat$, then $X$ is a $G$-poset and $i$ is both a Dwyer 
$G$-map and a relative $G\PI \Sd^2\aJ$-complex in $G\Pos$. 
The same statement holds for relative $G\PI \Sd^2\aI $-complexes. 
\end{cor}

We also have the equivariant analogue of \myref{retract}.

\begin{lem}\mylabel{retract2} A retract of a Dwyer $G$-map in $G\Pos$ is a Dwyer $G$-map.  Therefore 
all cofibrations in $G\Pos$ are Dwyer $G$-maps.
\end{lem}

We require a description of pushouts inside $G\Pos$.  The following is a
simplification of \cite[Lemma 2.5]{manyauthors}.  

\begin{lem}\mylabel{pushpos} 
  Let $j\colon K\rtarr L$ be a sieve of $G$-posets and $f\colon K\rtarr X$ be a
  map of $G$-posets. Consider the set $Y= (L\setminus K) \amalg X$ with the
  order relation given by restriction on $L\setminus K$ and on $X$, with the
  additional relation that for $x\in X$ and $y\in L\setminus K$, $x\leq y$ if there
  exists $w\in K$ such that $x \leq f(w)$ and $j(w) \leq y$. Then $Y$ is a $G$-poset and
  the following diagram is a pushout in $G\Pos$, where $k$ is the inclusion of
  the summand $X$ and $g$ is the sum of $f$ on $K$ and the identity on
  $L\setminus K$.
  \begin{equation}\label{push2} 
   \xymatrix{
    K \ar[d]_{j} \ar[r]^-{f} & X \ar[d]^{k} \\
    L \ar[r]_{g} & Y \\} 
  \end{equation}
  Moreover, if $j$ is a Dwyer map with factorization
  $\xymatrix@1{K\ar[r]^-{\io} & S\ar[r]^-{\nu} & L}$ and retraction $r\colon
  S\rtarr K$, then for $x\in X$ and $y\in L\setminus K$, $x \leq y$ if and only if $y = \nu(z)$ for some $z\in S$ such that $x \leq (f\com r)(z)$.
\end{lem}
\begin{proof}
  First, note that $Y$ is well-defined, since $L\setminus K$ is a $G$-subposet
  of $L$.  Indeed, if $y\in L\setminus K$ and $gy\in K$ then $y = g^{-1}g y \in
  K$, a contradiction.  The relation $\leq$ on $Y$ is reflexive and anti-symmetric
  since $L$ and $X$ are posets.  Transitivity requires a straightforward
  verification in the two non-trivial cases when $x \leq y$ and $y\leq z$ with either
  $x,y\in X$ and $z\in L\setminus K$ or $x\in X$ and $y,z\in L$. Thus $Y$ is a
  poset.

Clearly the map $k$ is order-preserving. Using that $j$ is a sieve, we see that $g$ is order-preserving 
by the definition of the order on $Y$. The square (\ref{push2}) is clearly a pushout of sets. Thus to show that
it is a pushout of posets it suffices 
to show that for any commutative square
\[ \xymatrix{
K \ar[d]_{j} \ar[r]^-{f} & X \ar[d]^{\ell} \\
L \ar[r]_{h} & Z \\} \]
of posets, the induced map $Y\rtarr Z$ is order-preserving. The only case that is non-trivial to check 
is when $x\leq y$ with $x\in X$ and $y\in L\setminus K$.  We must show that $\ell(x)\leq h(y)$. By assumption, 
there is an element $w\in K$ such that $x\leq f(w)$ and $j(w)\leq y$. It follows that
\[
\ell (x) \leq (\ell\com f)(w) = (h\com j)(w) \leq h(y  ),\]
as desired.  

For the last statement of the lemma, if $y=\nu(z)$ where $z\in S$ and $x\leq (f\com r)(z)$, let $w = r(z)$. Then
$x\leq f(w)$ and $j(w) = (\nu\com\io\com r)(z) \leq \nu(z) =y$ by the counit of the adjunction $(\io,r)$.  Conversely,
let $j(w)\leq y$ and $x\leq f(w)$. Since $\nu$ is a cosieve, $j(w) = (\nu\com\io)(w) \leq y$ implies $y = \nu(z)$ for some 
$z\in S$ with $\io(w) \leq z$, and then $w = (r\com \io)(w) \leq r(z)$ so that $x\leq f(w)$ implies $x \leq (f\com r\com \io)(w) \leq (f\com r)(z)$.
\end{proof}

Using this description we can show that pushouts along Dwyer $G$-maps are preserved
when taking $H$-fixed points for any subgroup $H$ of $G$.  The statement about
fixed points is a modification of \cite[Proposition 2.4]{manyauthors}.

\begin{lem}\mylabel{keytoomore} Let $j\colon K\rtarr L$ be a Dwyer $G$-map of $G$-posets, such as 
a retract of a relative $G\PI\Sd^2\aI$-cell complex, and let $f\colon K\rtarr X$
be any map of $G$-posets. Form the pushout diagram
\[ \xymatrix{
K \ar[d]_{j} \ar[r]^-{f} & X \ar[d] \\
L \ar[r] & Y \\} \]
in $G\Cat$. Then $Y$ is a $G$-poset and the diagram remains a pushout after taking $H$-fixed points 
for any subgroup $H$ of $G$.
\end{lem}
\begin{proof} Ignoring the $G$-action, the left vertical arrow is a Dwyer map of
  posets. Therefore $Y$ is a poset by \myref{cruxtoo}(iv) and is thus a
  $G$-poset. Fix a subgroup $H$ of $G$; by \myref{fixDwyer} $j^H$ is a Dwyer map, and thus
  the description from \myref{pushpos} can be used for $X^H\cup_{K^H} L^H$.
\end{proof}

\begin{exmp}
Let $G$ be the cyclic group of order two. Let $L$ be the three object $G$-poset depicted by 
$0\rtarr 2 \ltarr 1$ equipped with the action that interchanges $0$ and $1$, but fixes $2$.
Let $K$ be the $G$-subposet that consists of the elements $0$ and $1$. Then the inclusion $K\rtarr L$ is 
a sieve but {\em not} a Dwyer $G$-map. If $X=\ast$ is the terminal $G$-poset and $K\rtarr X$ is the unique map, 
then the pushout  $L\cup_K X$ in $G\Pos$ is the 
$G$-poset depicted by $*\rtarr 2$, with trivial $G$-action. Thus its $G$-fixed point poset is also
$*\rtarr 2$. However, the pushout $L^G\amalg_{K^G} X^G$ is the discrete poset with two elements 
$*$ and $2$.
\end{exmp}

\section{The proof of \myref{Main}}
For our equivariant model structures, we start with the following general 
result, which puts together results of the second author \cite[Proposition 2.6, Theorem 2.10]{stephan} with
augmentations of those results due to Bohmann, et al \cite[Propositions 1.4, 1.5, and 1.6]{manyauthors}, 
all reformulated in our simpler compactly generated setting.  Recall that $\aO_G$ 
denotes the orbit category of $G$.

\begin{defn} For a category $\aC$, let $G\aC$ denote the category of $G$-objects in $\aC$ and
let $\aO_G$-$\aC$ denote the category of contravariant functors $\aO_G\rtarr \aC$.  Assuming 
that $\aC$ has coproducts, define a functor
\[ \otimes\colon G\Set\times \aC \rtarr G\aC \]
by $S\otimes X = \amalg_S X$, the coproduct of copies of $X$ indexed by elements of $S$, 
with $G$-action induced from the action of $G$ on $S$ by permutation of the copies of $X$.
\end{defn}

We have an adjunction $(\LA,\PH)$ between $G\aC$ and $\aO_G$-$\aC$.  The left adjoint
$\LA$ sends a functor $\aO_G\rtarr \aC$ to its value on $G/e$ and the right adjoint $\PH$
sends a $G$-object to its fixed point functor.

\begin{thm}\mylabel{omni} Let $\aC$ be a compactly generated model category.  Assume that for each subgroup
$H$ of $G$, the $H$-fixed point functor $(-)^H\colon G\aC \rtarr \aC$ satisfies the
following properties. 
\begin{enumerate}[(i)]
\item It preserves colimits of sequences of maps $i_n\colon X_n\rtarr X_{n+1}$ in $G\aC$,
where each $i_n$ is a cofibration in $\aC$. 
\item It preserves coproducts.
\item It preserves pushouts of diagrams in which one leg is given by a coproduct of maps of the form
\[  \id\otimes j\colon G/J\otimes X \rtarr G/J\otimes Y, \]
where $j$ is a generating cofibration (or generating acyclic cofibration) of $\aC$ and J is a subgroup of $G$.\footnote{We don't
need to assume the condition for acyclic cofibrations, but we do so for convenience.}
\item For any object $X$ of $\aC$, the natural map 
\[ (G/J)^H\otimes X \rtarr (G/J\otimes X)^H \]
is an isomorphism in $\aC$.
\end{enumerate}
Then $G\aC$ admits a compactly generated model structure, where a map $f$ in $G\aC$ is a fibration or weak equivalence
if each fixed point map $f^H$ is a fibration or weak equivalence, so that $\PH(f)$ is a fibration
or weak equivalence in $\aO_G$-$\aC$. The generating (acyclic) cofibrations are the $G$-maps 
$\id\otimes j\colon G/J\otimes K \rtarr G/J\otimes L$, where the maps $j\colon K\rtarr L$ are the generating
(acyclic) cofibrations of $\aC$. Moreover, $(\LA,\PH)$ is then a Quillen equivalence between
$G\aC$ and $\aO_G$-$\aC$. Further, if $\aC$ is left or right proper, then so is $G\aC$.
\end{thm}

By \cite[1.3]{manyauthors}, the model structure is functorial with respect to Quillen pairs.

\begin{thm} Let $\aC$ and $\aD$ be compactly generated model categories satisfying the assumptions of
\myref{omni} and let $(L,R)$ be a Quillen pair between them.  Then there is an induced Quillen
pair between $G\aC$ and $G\aD$, and it is a Quillen equivalence if $(L,R)$ is a 
Quillen equivalence.
\end{thm} 

\begin{proof}[Proof of \myref{Main}]
We need only verify conditions (i)-(iv) of
  \myref{omni} when $\aC = \Pos$.  Cofibrations in $\Pos$ are inclusions and if
  $x\in X = \colim X_n$, then $x\in X^H$ if and only if $x\in X_n^H$ for a large
  enough $n$; thus condition (i) holds. Condition (ii) holds by the definition
  of coproducts in $\Cat$.  Since the action of $G$ on $G/J\otimes X$ comes from
  the action of $G$ on $G/J$, condition (iv) holds as well.

  It remains to check condition (iii).  By \myref{thomcrux}, the generating (acyclic)
  cofibrations in $\Pos$ are Dwyer maps.  Consider a pushout
  diagram in $G\Cat$
  \[ \xymatrix{
    \coprod_{i\in I} G/J_i\otimes K_i \ar[d]_{\amalg \id\otimes j_i}
    \ar[r]^-{\coprod f_i} & X \ar[d] \\
    \coprod_{i\in I} G/J_i\otimes L_i \ar[r] & Y \\} \]
  where each $j_i:K_i \rtarr L_i$ is a Dwyer map and $f_i:G/J_i \otimes
  K_i \rtarr X$ is a map of $G$-posets.  Condition (iii) holds if, for any such diagram, $Y$ is a $G$-poset
  (hence $Y^H$ is also a poset) and the diagram remains a pushout after passage to $H$-fixed points.
   This is a special case of
 \myref{keytoomore}.  
\end{proof}

\section{Cofibrant posets}\label{Cof}

Since every cofibrant object in $\Cat$ is a poset and, by \myref{cofibrations}, a poset is cofibrant in $\Pos$ if and only if it is cofibrant in $\Cat$, 
it follows that $\Pos$ and $\Cat$ have the same cofibrant objects.  We have an explicit cofibrant replacement functor for $\Pos$, namely double 
subdivision.  While this does give a large class of cofibrant objects, it does not help to determine whether or not a given poset is cofibrant.
By \myref{retract}, any cofibration in $\Pos$ is a Dwyer map and it follows immediately from the definition of Dwyer maps 
that the map $\emptyset \rtarr P$ is a Dwyer map for any poset $P$. Our understanding is summarized in the following picture:
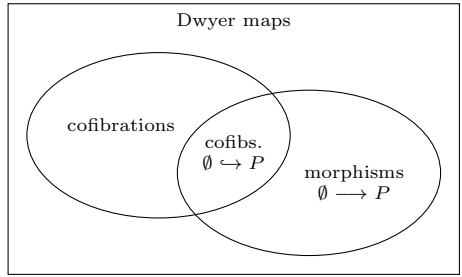
\begin{center}
  \medskip
  \begin{tikzpicture}[font=\scriptsize]
    \draw (0,0.4) rectangle (6,4); \node[below] at (3,4) {Dwyer maps}; \draw
    (2,2.25) ellipse (1.75 and 1.1); \node[below] at (1.5,2.6) {cofibrations};
    \draw (4,1.75) ellipse (1.75 and 1.1); \node[above] at (4.6,1.2)
    {\begin{tabular}{c}morphisms \\ $\emptyset \rtarr P$\end{tabular}}; \node at
    (3,2) {\begin{tabular}{c}cofibs. \\$\emptyset \hookrightarrow
        P$\end{tabular}};
  \end{tikzpicture}
  \medskip
\end{center}
It is not difficult to show that most of the sections in this Venn diagram are
nonempty; the only difficulty is to show that there exist morphisms
$\emptyset \rtarr P$ which are not cofibrations. As the referee pointed out to us, 
it is not hard to find infinite posets that are not cofibrant, such as the natural numbers 
with its reverse ordering.  However, as far as we know ours is the first example of a finite
poset that is not cofibrant. Specifically, in Proposition~\ref{2sphere} we show that the following model of the $2$-sphere, 
which is a finite poset $A$ whose classifying space is homeomorphic to $S^2$,
is not cofibrant in $\Pos$.
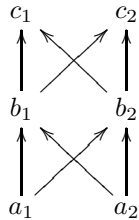
\[\xymatrix{ c_1 & c_2 \\ b_1 \ar[ru] \ar[u] & b_2 \ar[lu] \ar[u] \\ a_1 \ar[ru]
  \ar[u] & a_2 \ar[lu] \ar[u]}\]
  
  This example of a finite, non-cofibrant poset is minimal in dimension and in cardinality. We prove in Proposition~\ref{onedimensionalcofibrant} that every one-dimensional finite poset is cofibrant, and 
Bruckner and Pegel \cite{BP}  have shown that every poset with at most five elements is cofibrant.

We first give a tool for showing that posets are not cofibrant.
\begin{lem}\mylabel{retractcondition}
  Let $A$ be a nonempty finite poset. Suppose that $A$ satisfies the following
  condition: for any pushout square
  \[\xymatrix{\Pi \Sd^2\partial \Delta[n] \ar[r] \ar[d] & X \ar[d] \\ 
    \Pi \Sd^2 \Delta[n] \ar[r] & Y}\]
  if $A$ is a retract of $Y$ then it is also a retract of $X$.  Then $A$ is not
  cofibrant in $\Pos$.
\end{lem}
\begin{proof}  Assume that  $A$ is cofibrant. We prove that $A$ must be
empty, a contradiction.  Since $\Pos$ is compactly generated and $A$ is cofibrant, $A$ is a
  retract of a sequential colimit $\colim_n X_n$, where $X_0=\emptyset$ and
  $X_i\rtarr X_{i+1}$ is a pushout of a coproduct of generating cofibrations for
  $i\geq 0$. Since $A$ is finite, the inclusion $A\rtarr \colim_n X_n$
  factors through some $X_n$, and then $A$ is a retract of $X_n$. Assume $n>0$.  Since $A$ is finite, the inclusion $A\rtarr X_n$
  factors through a pushout $Y_n$ obtained by attaching only finitely many
  generating cofibrations to $X_{n-1}$, and then $A$ is a retract of $Y_n$. We can now use the assumed condition on
  $A$ to induct downwards one generating cofibration at a time; our condition
  ensures that $A$ is a retract of $X_{n-1}$.  Iterating, we deduce that $A$ is
  a retract of $X_0=\emptyset$ and thus $A=\emptyset$.
\end{proof}

We will also need the following explicit description of the generating cofibrations
\[\Pi \Sd^2\pa \DE[n] \rtarr \Pi\Sd^2\DE[n].\]
An element of the poset $\PI\Sd^2 \DE[n]$ is a sequence of strict inclusions
\[S_0\subset \ldots \subset S_k \] of nonempty subsets of
$\n=\{0,\ldots,n\}$. We can identify such a sequence with the totally ordered set $\{S_0,\ldots,S_k\}$. With this
identification the order relation on $\PI\Sd^2 \DE[n]$ is given by subset
inclusion.  The poset $\PI\Sd^2 \pa\DE[n]$ is the subposet of $\PI\Sd^2
\DE[n]$ given by the sequences $S_0\subset \ldots \subset S_k$ with $S_k\neq
\n$.

We are now ready to show that our model $A$ of the $2$-sphere is not cofibrant.
\begin{prop} \mylabel{2sphere}
The finite poset $A$ is not cofibrant in $\Pos$.
\end{prop}
\begin{proof}
  We will show that $A$ satisfies the condition in
  Lemma~\ref{retractcondition}; since $A$ is nonempty, this implies that $A$ is
  not cofibrant.

  Let $Y$ be the pushout of a diagram of the form
  \[\PI\Sd^2\DE[n] \ltarr \PI\Sd^2\pa \DE[n]\rtarr X,
  \]
  where $X$ is any poset. We use the explicit description of the pushout from \myref{pushpos}. Suppose that $A$ is a retract of $Y$, so that
  $\id_A$ admits a factorization $A \stackrel{i}{\rtarr}  Y \stackrel{r}{\rtarr} A$.

  Consider the map $(\PI\Sd^2\DE[n])\setminus\{\n\} \rtarr \PI\Sd^2\pa \DE[n]$
  defined by
  \[S_0\subset \ldots \subset S_k\quad \longmapsto  \begin{cases}
    S_0\subset \ldots \subset S_{k-1} & \hbox{if } S_k=\n, \\
    S_0\subset \ldots \subset S_k &\text{ otherwise.}
  \end{cases}
  \]
  This induces a map $p\colon Y\setminus\{\n\}\rtarr X$. We show that $\n\notin
  i(A)$, and that the composite
  \begin{equation}\label{ret} A \stackrel{i}{\rtarr} Y\setminus \{\n\} \stackrel{p}{\rtarr} X \rtarr Y
  \stackrel{r}{\rtarr} A
  \end{equation}
  is the identity on $A$.  From this we can conclude that $A$ is a retract of
  $X$.

  Since $\n\in Y$ is not a codomain of a non-identity arrow, the only elements
  of $A$ that $i$ could send to $\n$ are $a_1$ and $a_2$. We show more
  generally, that $i(a_1),i(a_2)\in X$. If $i(a_1)\in Y\setminus X$ or
  $i(a_2)\in Y\setminus X$, then $i(b_1),i(b_2),i(c_1),i(c_2)\in Y\setminus
  X$. Considering $i(c_1)$ and $i(c_2)$ as totally ordered sets of nonempty
  subsets of $\n$, the intersection $i(c_1)\cap i(c_2)$ is an element of
  $Y\setminus{X}$ and we have a diagram
  
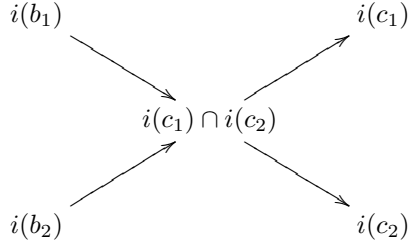
\[\xymatrix{ i(b_1)\ar[rd] & & i(c_1) \\
    & i(c_1)\cap i(c_2)\ar[ru]\ar[rd] & \\
    i(b_2) \ar[ru] & & i(c_2)
  }
  \]
  in $Y$. Applying the retraction $r\colon Y\rtarr A$ to this diagram yields an
  arrow between $b_1$ and $b_2$ or an arrow between $c_1$ and $c_2$. Both cases
  are impossible. We have shown that $i(a_1),i(a_2)\in X$ and thus that
  $\n\notin i(A)$.

  We can also show by the same argument as above that $i(b_1)$ and $i(b_2)$
  cannot both belong to $Y\setminus X$.

  It remains to show that the composite (\ref{ret}) is the identity.
  Recall that $i(a_1)$, $i(a_2)$ and at least one of $i(b_1)$, $i(b_2)$ belong
  to $X$. By symmetry we can assume that $i(b_2)\in X$. We need to show that
  $rpi(b_1)=b_1$, $rpi(c_1)=c_1$ and $rpi(c_2)=c_2$. Implicitly, we will use
  that any arrow in $Y$ from an element in $X$ to an element $z$ in
  $Y\setminus X$ factors through $p(z)$. Since $i(a_1) \leq i(b_1)$ and $i(a_2) \leq
  i(b_1)$, we have a diagram
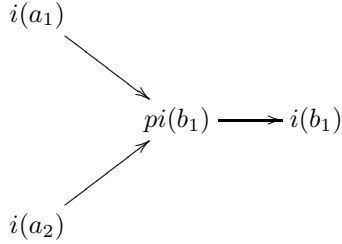
\[\xymatrix{i(a_1)\ar[rd] & & \\
& pi(b_1)\ar[r] & i(b_1) \\
i(a_2)\ar[ru] & &}
\]
in $Y$. By applying $r$ to this diagram, we deduce that $rpi(b_1)=b_1$ since there is no arrow between $a_1$ and $a_2$.

Applying $r$ to the diagram
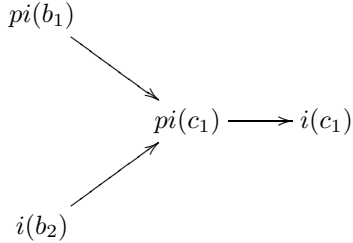
\[\xymatrix{pi(b_1)\ar[rd] & & \\
& pi(c_1)\ar[r] & i(c_1) \\
i(b_2)\ar[ru] & &}
\]
in $Y$, we deduce that $rpi(c_1)=c_1$. By symmetry, we also have $rpi(c_2)=c_2$. We have shown that $A$ is a retract of $X$.
\end{proof}

\begin{cor}
Not all finite posets in Thomason's model structure on $\Cat$ are cofibrant.
\end{cor}

The above proof used many special properties of $A$ and thus cannot be used in
general to determine which objects are cofibrant.  However, there is one
class of posets that we can prove are cofibrant: the one-dimensional finite ones.
We say that a poset $P$ is (at most) one-dimensional if in any pair of composable 
morphisms at least one is an identity morphism.  

\begin{prop}\mylabel{onedimensionalcofibrant}
Every one-dimensional finite poset $X$ is cofibrant.
\end{prop}

\begin{proof} We proceed by induction on the number $m$ of elements of $X$.
If $m=0$, then $X=\emptyset$ and is thus cofibrant. Now suppose that $m\geq 1$. 
If $X$ has no non-identity morphisms (is zero-dimensional), then 
$X$ can be built up by attaching singleton sets $\PI\Sd^2\DE[0]$ to $\emptyset$ and is thus cofibrant.

  Otherwise, let $a$ be the domain of a non-identity morphism. Set
  $A=X\setminus{\{a\}}$. By the induction hypothesis $A$ is cofibrant. Let
  $Y=\{y_0,\ldots,y_n\}$ be the set of elements $y\in X$ such that there exists
  a non-identity morphism $a \rtarr y$ in $X$.  

  Let $CY$ denote the cone on $Y$ obtained by adding a least element $*$ to
  $Y$. Note that $Y\rtarr CY$ is an inclusion of a cosieve. Thus $X\cong
  A\cup_Y CY$ by a dual version of \myref{pushpos}. 
  
  We distinguish the two cases $n=0$ and $n>0$. If $n=0$, we glue $\PI\Sd^2\DE[1]$ to $A$ along a cofibration in such a way that $X$ is a retract of the resulting pushout, and therefore cofibrant. The inclusion of the vertex $0$ into $\DE[1]$ is a cofibration. Applying $\PI\Sd^2$ to this cofibration yields the inclusion of the poset $\{\{0\}\}$ into $\PI\Sd^2\DE[1]$. Identifying the element $\{0\}$ with $y_0$, we show that $X$ is a retract of the pushout $A\cup_Y \PI\Sd^2\DE[1]$. Let $X\rtarr A\cup_Y \PI\Sd^2\DE[1]$ be the map
   \[x\mapsto \begin{cases} \{0\}\subset \mathbf{1} &  \text{ if } x=y_0\\
    \mathbf{1} &  \text{ if } x=a\\
    x & \text{ otherwise}
  \end{cases}
  \]
 The map $\PI\Sd^2\DE[1]\rtarr CY$,
 \[S_0\subset \ldots \subset S_k\quad \mapsto \begin{cases}
    y_0 &  \text{ if } S_0=\{0\} \\
    * &  \text{ otherwise }
  \end{cases}
  \]
  induces a retraction $A\cup_Y \PI\Sd^2\DE[1]\rtarr X$ of the map $X\rtarr A\cup_Y \PI\Sd^2\DE[1]$ above. Thus $X$ is cofibrant if $n=0$ and we now assume that $n>0$.
  
  Similarly to the case $n=0$, we glue $\PI\Sd^2\DE[n]$ to $A$ along a cofibration in such a way that $X$ is a retract of the resulting pushout, and therefore cofibrant.
  
  The inclusion of the set of vertices of $\DE[n]$ into $\DE[n]$ is a
  cofibration. Applying $\PI\Sd^2$ to this cofibration yields the inclusion of
  the discrete poset $\{\{i\}\,|\,0\leq i\leq n\}$ into
  $\PI\Sd^2\DE[n]$. Identifying the element $\{i\}$ with $y_i$, let $Z$ denote
  the pushout $A\cup_Y (\PI \Sd^2\DE[n])$.  We claim that $X$ is a retract
  of $Z$. Indeed, let $j\colon X\rtarr Z$ be the map
  \[x\mapsto \begin{cases} \{i\}\subset \n & \text{ if } x=y_i\\
    \n & \text{ if } x=a\\
    x & \text{ otherwise}
  \end{cases}
  \]
  The map $\PI\Sd^2\DE[n]\rtarr CY$, 
  \[S_0\subset \ldots \subset S_k\quad \mapsto \begin{cases}
    y_i & \text{ if } S_0=\{i\} \\
    * & \text{ otherwise}
  \end{cases}
  \]
  induces a map $r\colon Z\rtarr X$ such that $rj=\id_X$ as desired.
\end{proof}

We illustrate the induction step of this proof using the following poset $X$:
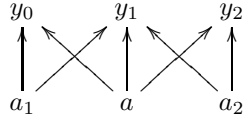
\[\xymatrix{y_0 & y_1 & y_2 \\ a_1 \ar[u] \ar[ru] & a \ar[lu] \ar[u] \ar[ru] &
 a_2 \ar[lu] \ar[u]}\]
After removing $a$ we obtain the following poset $A$, which by induction hypothesis is cofibrant.
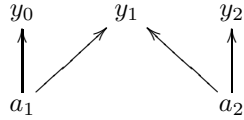
\[\xymatrix{y_0 & y_1 & y_2 \\ a_1 \ar[u] \ar[ru] &  &
  a_2 \ar[lu] \ar[u]}\]
The poset $Z$ in the proof above can be pictured as follows. 
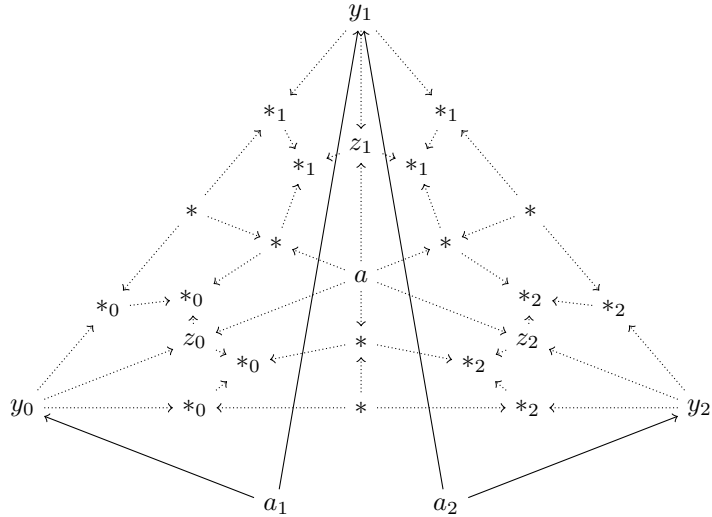
\begin{center}
  \tdplotsetmaincoords{30}{0}
  \begin{tikzpicture}[tdplot_main_coords,yscale=3.5, xscale=4.5, baseline=0]
    \node (A1) at (0,0,0) {$y_0$};
    \node (A2) at (1,{sqrt(3)},0) {$y_1$};
    \node (A3) at (2,0,0) {$y_2$};
    
    \node (A12) at (barycentric cs:A1=1,A2=1,A3=0) {$*$};
    \node (A13) at (barycentric cs:A1=1,A2=0,A3=1) {$*$};
    \node (A23) at (barycentric cs:A1=0,A2=1,A3=1) {$*$};
    \node (A123) at (barycentric cs:A1=1,A2=1,A3=1) {$a$};

    \arDiv[$*_0$]{1}{12}
    \arDiv[$*_0$]{1}{13}
    \arDiv[$*_1$]{2}{12}
    \arDiv[$*_1$]{2}{23}
    \arDiv[$*_2$]{3}{13}
    \arDiv[$*_2$]{3}{23}
    \arDiv{12}{123}
    \arDiv{13}{123}
    \arDiv{23}{123}
    \arDiv[$z_0$]{1}{123}
    \arDiv[$z_1$]{2}{123}
    \arDiv[$z_2$]{3}{123}

    \baryc[$*_0$]{1}{12}{123}
    \baryc[$*_0$]{1}{13}{123}
    \baryc[$*_1$]{2}{12}{123}
    \baryc[$*_1$]{2}{23}{123}
    \baryc[$*_2$]{3}{13}{123}
    \baryc[$*_2$]{3}{23}{123}
    
    \path (A12-123) -- ++(0,0,-2) node (B1) {$a_1$};
    \path (A23-123) -- ++(0,0,-2) node (B2) {$a_2$};

    \draw[->] (B1) -- (A1);
    \draw[->] (B1) -- (A2);
    \draw[->] (B2) -- (A2);
    \draw[->] (B2) -- (A3);
  \end{tikzpicture}
\end{center}
Here each vertex is a distinct object of $Z$ (although we have not given the objects distinct names),
and the edges give all of the non-identity morphisms of $Z$. The inclusion $j\colon X\rtarr Z$ maps 
$a_i$ to $a_i$, $y_k$ to $z_k$ and $a$ to $a$.  The retraction $r$ is defined by
\[
\begin{array}{ccc}
  r(z_k) = r(y_k) = r(*_k) = y_k \qquad  r(a_i) = a_i \qquad
  r(a) = r(*) = a.
\end{array}\]

The essential point is that, even in such simple cases as in this section, proving that a poset is or is not cofibrant is a non-trivial exercise.

\section{Fibrant posets}\label{Fib}

In this section we give a class of examples of fibrant posets.  Before we begin we
give several easy lemmas needed in the proofs.
First, we show that when proving a category is fibrant it suffices to consider its
connected components.  Here, a category is connected if any two objects are connected 
by a finite zigzag of morphisms.   A component of a category is a maximal connected full subcategory, 
and any category is the disjoint union of its components.

\begin{lem} \label{conncomp} Let $\C\in \Cat$ or $\Pos$.  Then $\C$ is fibrant
if and only if all of its components are so.
\end{lem}

\begin{proof} The image of a connected category under a functor lies in a
  single component.  Since each $\Pi \Sd^2\Lambda^k[n]$
  is connected, any functor 
  $\Pi \Sd^2 \Lambda^k[n] \rtarr \C$ lands in a single component and similarly
  for $\Pi \Sd^2 \Delta[n]$.  A category $\C$ is fibrant if and only if for every functor
  $f:\Pi\Sd^2\Lambda^k[n] \rtarr \C$, there exists a functor $h\colon \Pi \Sd^2 \Delta[n]
  \rtarr \C$ such that the diagram
  \[\xymatrix{ \Pi\Sd^2 \Lambda^k[n] \ar[rr]^-f \ar[d] && \C \\ \Pi \Sd^2
   \Delta[n] \ar[rru]_-h}\]
  commutes, and this holds if and only if it holds with $\C$ replaced by each of its components.
\end{proof}

Second, we record the following result relating pullbacks and pushouts
to binary products $\times$ and binary coproducts $\cup$ inside a poset $P$.
Its proof is an exercise using that there is at most one morphism between any two 
objects of $P$.

\begin{lem} \label{lem:prodcoprod}  If the pullback of a given pair of maps 
$x\rtarr a\longleftarrow y$ exists, it is the product $x\times y$, and if the product 
$x\times y$ exists, it is the pullback of any pair of maps $x\rtarr a\longleftarrow y$. 
Dually, if the pushout of a given pair of maps 
$x\longleftarrow a \longrightarrow y$ exists, it is the coproduct $x\cup y$, and if the coproduct 
$x\cup y$ exists, it is the pushout of any pair of maps $x\longleftarrow a \longrightarrow y$.
\end{lem}

The following addendum implies that a poset with binary products or coproducts is contractible, 
meaning that its classifying space is contractible.

\begin{lem} \label{lem:c->cont}
If $P$ is a poset containing an object $c$ such that either  $c\times x$ exists for any 
$x\in P$ or $c\cup x$ exists for any $x\in P$, then $P$ is contractible.
\end{lem}
\begin{proof}
We prove the lemma in the first case; the second case follows by duality.  Let
$P/c$ be the poset of all elements $x$ over $c$; this means that $x\leq c$, or, thinking of 
$P$ and $P/c$ as categories, that there
is a morphism $x \rtarr c$; it is contractible
since it has the terminal object  $c \rtarr c$. 
Since $P$ is a poset, there is at most one morphism $x\rtarr c$ for any object $x$ 
and the functor $P\rtarr P/c$ that sends an object $y$ to $c\times y \rtarr c$ is right adjoint 
to the forgetful functor that sends $x \rtarr c$ to $x$.  Therefore the classifying space
of $P$ is homotopy equivalent to that of $P/c$. 
\end{proof}

In \cite{MO}, Meier and Ozornova construct examples of fibrant categories.
They start from the notion of a partial model category, which is a weakening 
of the notion of a model category.  Recall that a  homotopical category $(\C,\W)$ 
is a category $\C$ together with a subcategory $\W$, whose maps we call weak
equivalences, such that every object of $\C$ is in $\W$ and $\W$ satisfies the
$2$ out of $6$ property: if morphisms $h\com g$ and $g\com f$ are in $\W$,
then so are $f$, $g$, $h$, and $h\com g\com f$.

\begin{defn}[{\cite[\S1.1]{BK}}]
  A \textsl{partial model category} is a homotopical category $(\C,\W)$ such that $\W$ 
 contains subcategories $\mathcal{U}$ and $\mathcal{V}$ that satisfy the following
properties.
  \begin{enumerate}[(i)]
   \item  $\mathcal{U}$ is closed under pushouts along morphisms in $\C$
    and $\mathcal{V}$ is closed under pullbacks along morphisms in $\C$.
  \item   The morphisms of $\W$ admit a functorial factorization into a
    morphism in $\mathcal{U}$ followed by a morphism in $\mathcal{V}$.
  \end{enumerate}
\end{defn}

In (i), it is implicitly required that the cited pushouts and pullbacks exist in $\C$.
For example, if $\C$ has a model structure with weak equivalences $\W$ then
it has a partial model structure, with $\mathcal{U}$ being the
subcategory of acyclic cofibrations and $\mathcal{V}$ being the subcategory of
acyclic fibrations.

\begin{thm}[{\cite[Main Theorem]{MO}}] \label{cor:MO} If $(\C,\W)$ is a homotopical 
category that admits a partial model structure, then
$\W$ is fibrant in the Thomason model structure on $\mathbf{Cat}$. 
\end{thm}

In the present context, it is very natural to consider those partial model
structures such that $\C$ is a poset.  In \cite{DZ}, Droz and Zakharevich
classified all of the model structures on posets.

\begin{thm}[{\cite[Theorem B]{DZ}}] \label{thm:DZ} Let $P$ be a poset which
  contains all finite products and coproducts, and let $\W$ be a subcategory
  that contains all objects of $P$.  Then $P$ has a model structure with $\W$ as
  its subcategory of weak equivalences if and only if the following two
  properties hold.
   \begin{enumerate}[(i)]
  \item If a composite $gf$ of morphisms in $P$ is in $\W$, then both $f$ and
    $g$ are in $\W$.
  \item There is a functor $\chi\colon P \rtarr P$ that takes all maps in $\W$ 
  to identity maps and has the property that for every object $x\in P$, the four canonical maps of the diagram
    \[\xymatrix{\chi(x)\times x \ar[r] \ar[d] & \chi(x) \ar[d] \\ x \ar[r] &  \chi(x)\cup x}\]
    in $P$ are weak equivalences.
  \end{enumerate}
\end{thm}

These two results have the following consequence.

\begin{prop} \label{prop:ex} Let $P$ be a poset satisfying the following
  conditions:
  \begin{enumerate}[(i)]
  \item $P$ contains an object $c$ such that $c\times x$ and $c\cup x$ exist in
    $P$ for any other object $x\in P$.
  \item For any two objects $a,b\in P$, either $a\times b$ exists or there does
    not exist an $x\in P$ such that $x \leq a$ and $x\leq b$.  Dually, either $a
    \cup b$ exists or there does not exist an $x\in P$ such that $x \geq a$ and $x
    \geq b$.
  \end{enumerate}
  Then $P$ is a component of the weak equivalences in a model category
  and is therefore fibrant in $\Pos$.  Moreover, $P$ is contractible.
\end{prop}

\begin{proof}
  Consider the poset $\tilde{P}$ whose objects are those of $P$ and two further
  objects, $\emptyset$ and $\ast$.  The morphisms are those of $P$ and those
  dictated by requiring $\emptyset$ to be an initial object and $\ast$ to be a
  terminal object (so there is no morphism $\ast \rtarr \emptyset$). Condition (ii)
  ensures that $\tilde{P}$ has all finite products and coproducts.  Indeed, if $a,b\in P$ and 
  $a\times b$  does not exist in $P$, then $a\times b = \emptyset$ in $\tilde{P}$ and dually
  for coproducts. For all $x\in \tilde P$, $x\times \ast = x$, $x\times \emptyset = \emptyset$,
  $x\cup \emptyset = x$, and  $x\cup \ast =\ast$.
  
  Let $\W$ be the union of $P$ and the discrete subcategory $\{\emptyset, *\}$ of $P$.
  Although $\tilde P$ is connected, $P$ is one of the three components
  of $\W$, the other two being the discrete components $\{\emptyset\}$ and $\{*\}$ (which are clearly fibrant). 
  Theorem~\ref{thm:DZ} implies that $\tilde P$ has a model structure with $\W$ as its
  subcategory of weak equivalences. Indeed, condition (i) is clear and, for condition (ii),
  we define $\chi\colon \tilde P\rtarr \tilde P$ by mapping all of $P$ to $c$ (and
  its identity morphism), mapping $\emptyset$ to $\emptyset$, and mapping $\ast$
  to $\ast$.  Therefore $\W$ is fibrant by Theorem~\ref{cor:MO}, hence $P$ is fibrant by
  Lemma~\ref{conncomp}; $P$ is contractible by Lemma \ref{lem:c->cont}.
\end{proof}

For example, if $P$ and $Q$ are any posets satisfying condition (ii) of
Proposition~\ref{prop:ex} then the following poset is fibrant:
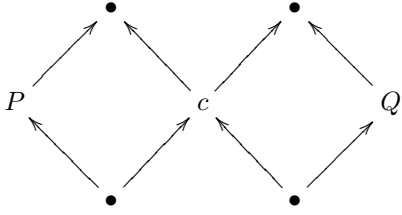
\[
  \xymatrix{
    & \bullet & & \bullet \\
    P \ar[ru] & & c \ar[ru] \ar[lu] & & Q \ar[lu] \\
    & \bullet \ar[lu] \ar[ru] & & \bullet \ar[lu] \ar[ru] \\
  }  
\]

Finally, we prove a partial converse to Proposition \ref{prop:ex} which shows that in many cases the connected
fibrant posets constructed by Theorem~\ref{cor:MO} are contractible. 

\begin{defn} A map $f\colon a \rtarr b$ in a poset $P$ is {\em maximal} if there
do not exist any non-identity morphisms $ z\rtarr a$ or $b\rtarr z$.    
\end{defn} 
For example, the 
 composition of a sequence of maximal length in $P$ is maximal.
 
 \begin{prop} \label{prop:partial-model} Let $( \W, \mathcal{U}, \mathcal{V})$
   be a partial model structure on a poset $P$ and let $Q$ be a connected
   component of $\W$ that contains a maximal map.  Then $Q$ contains an object
   $c$ such that $c\times x$ and $c\cup x$ exist in $Q$ for any other object
   $x\in Q$. Therefore $Q$ is contractible.
\end{prop}

\begin{proof}
  Let $f\colon a \rtarr b$ be a maximal map in $Q$ and factor it as a map $a
  \rtarr c$ in $\sU$ followed by a map $c \rtarr b$ in $\sV$, using the
  functorial factorization.  Since $Q$ is a connected component of $\W$, $c$ is
  in $Q$.
  
  First, we claim that any morphism $g\colon z \rtarr c$ in $Q$ is in
  $\mathcal{U}$.  Factor $g$ as a morphism $z \rtarr w$ in $\mathcal{U}$
  followed by a morphism $w \rtarr c$ in $\mathcal{V}$.  Since $\mathcal{V}$ is
  closed under pullbacks, $a\times_c w\rtarr a$ exists and is in $\mathcal{V}$.
  However, since $f$ is maximal in $\W$, we must have $a \times_c w = a$, so
  there exists a morphism $a \rtarr w$.  By Lemma~\ref{lem:prodcoprod}, the
  pushout $c\cup_a w$ of $a\rtarr c$ along $a \rtarr w$ is $w \cup c$, and
  $w\cup c = c$ since $P$ is a poset and there is a map $w\rtarr c$.  But then
  $w \rtarr c$ is the pushout of a morphism in $\mathcal{U}$, so it is also in
  $\mathcal{U}$.  Thus $g$ is the composite of two morphisms in $\mathcal{U}$,
  so it is also in $\mathcal{U}$, as claimed.
 Dually, any morphism $c\rtarr z$ in $Q$ is in $\mathcal{V}$.  

  Now let $x$ be any object in $Q$.  Since $Q$ is connected, there is a
  finite zigzag of morphisms of $Q$ connecting $x$ to $c$. 
 If the zigzag ends with 
 \[ \xymatrix@1{ w\ar[r]^-h & y  & \ar[l]_-{i} z \ar[r]^-{j} & c,\\} \]
 then $j$ is in $\mathcal{U}$, so $y \cup_{z} c$ exists and we can shorten
 the zigzag via the diagram 
 \[  \xymatrix{ w\ar[r]^-h \ar[drr] & y \ar[dr] & \ar[l]_{i} z \ar[r]^{j} & c  \ar[dl] \\
& & y\cup_z c. & \\} \]
The dual argument applies to shorten the zigzag if it ends with 
\[ \xymatrix@1{ w & \ar[l]_-{h} y   \ar[r]^-{i}  & z &  \ar[l]_-{j} c.\\} \]
 Inductively, we can shorten any zigzag to one of either of the forms 
  \[\xymatrix@1{x & \ar[l] z \ar[r] & c &  \ \hbox{or} \  & x \ar[r] & z & \ar[l] c.\\} \]
 We show that $c\cup x$ and $c\times x$ exist in the first
  case; the same is true in the second case by symmetry.  Since $z \rtarr c$ is in $\mathcal{U}$, 
  $c\cup_z x$ exists, and it is $c\cup x$ by Lemma~\ref{lem:prodcoprod}. Since $c\rtarr c\cup x$ is in $\sV$, 
  $c \times_{c\cup x} x$ also exists, and it is $c\times x$ by
  Lemma~\ref{lem:prodcoprod} again.

  Thus $Q$ contains an object $c$ such that $c\times x$ and $c\cup x$ exists for
  any object $x\in Q$, as claimed, and it follows from Lemma~\ref{lem:c->cont} that  $Q$ is
  contractible.
\end{proof}

\section*{Acknowledgement}
We thank Viktoriya Ozornova for pointing out some misleading typos and a mistake in our original proof 
of \myref{onedimensionalcofibrant} and for alerting us to the work of Bruckner and Pegel. We are grateful to the referee for detailed comments and, in particular, for pointing out that infinite non-cofibrant posets are known. 
\bibliographystyle{plain} 
\bibliography{PMSTIZG}

\end{document}